\documentclass[12pt]{article}

\usepackage{amsfonts}
\usepackage{amsxtra}
\usepackage{amsthm}
\usepackage{amssymb}
\usepackage{amsmath,amscd}
\usepackage{supertabular}
\usepackage{graphics}
\usepackage{bm}

\newcommand{\vect}[1]{\mathbf{#1}}
\newcommand{\abs}[1]{{\mathopen\mid}{#1}{\mathclose\mid}}

\newtheorem{theorem}{Theorem}[section]

\newtheorem{remark}[theorem]{Remark}
\newtheorem{corollary}[theorem]{Corollary}

\makeindex

\begin{document}

\title{Exercises on the Kepler ellipses through a fixed point in space,
after Otto Laporte}

\author
{Gert Heckman\\
Radboud University Nijmegen}

\date{Dedicated to Tom Koornwinder on the occasion \\ of his 80th birthday}

\maketitle

\begin{abstract}
This article has a twofold purpose. On the one hand I would like to draw 
attention to some nice exercises on the Kepler laws, due to Otto Laporte 
from 1970. Our discussion here has a more geometric flavour than the 
original analytic approach of Laporte. 

On the other hand it serves as an addendum to a paper of mine from 1998  
on the quantum integrability of the Kovalevsky top. Later  I learned 
that this integrability result had been obtained already long before by 
Laporte in 1933.
\end{abstract}

\section{Introduction}
\label{Introduction}

In the first decade of this century Maris van Haandel and I taught
for several years a master class for high school students on the 
Kepler laws of planetary motion. 
The proof that the orbits of the planets are ellipses is usually given 
by clever calculus tricks, which might leave the innocent student 
with a feeling of black magic, although opinions can differ.
For example, Herbert Goldstein describes this proof as the
``simplest way to integrate the equation for the orbit", see
Section 3.7 of his excellent text book on classical
mechanics \cite{Goldstein 1980}.

In the preparation of our master class we found a proof, that
was more geometric in nature, and based on the focus-focus 
characterization of ellipses \cite{van Haandel--Heckman 2009}.
After the standard initial discussion in Section 2 of the conservation 
laws of angular momentum and total energy, and their consequences
for the Kepler problem, our proof will be recalled in the Section 3. 
An elegant alternative geometric proof based on the focus-directrix 
characterization of ellipses was given by Alexander Givental 
\cite{Givental 2016}. Several other proofs, like the original one
of Isaac Newton from 1687 and the one by Richard Feynman
from 1964, were discussed in modern mathematical language
in \cite{van Haandel--Heckman 2009}.

Recently I became aware of a paper by Otto Laporte on some
geometric properties of the Kepler ellipses through a fixed 
point in space \cite{Laporte 1970}. His results were obtained
while teaching classical mechanics during numerous years
in order to provide interesting exercises for students learning
the mathematics of the Kepler laws. His analytic results will be
conveniently derived in a geometric way in Section 4. 

The final Section 5 serves as an addendum to an old paper of mine  
on the quantization of the Kovalevsky top \cite{Heckman 1998}.

I would like to thank Rainer Kaenders and the anonymous referee 
for useful comments.

\section{The familiar conservation laws}
\label{The familiar conservation laws}

Let $\vect{r}$ be the radius vector of a point in $\mathbb{R}^3$
and let the scalar $r$ denotes its length. If $\vect{r}$ moves in
time $t$ then $\dot{\vect{r}}$ denotes its velocity and
$\ddot{\vect{r}}$ its acceleration. As usual the dot always stands 
for the derivative with respect to time. The \emph{Kepler problem} 
studies the solutions of Newton's equation of motion
\[ \mu\ddot{\vect{r}}=\vect{F} \]
for an inverse square force field  $\vect{F}=-k\vect{r}/r^3$  defined
on $\mathbb{R}^3$ minus the origin.  The vector $\vect{r}$ 
describes the relative motion of a particle with mass $m$ around 
another particle with mass $M$. The parameter $\mu=mM/(m+M)$ 
is called the reduced mass and $k=GmM$ the coupling constant,
with $G$ Newton's universal gravitational constant.

The second law of Kepler that the motion is planar and that the 
radius vector traces out equal areas in inequal times is easy to prove. 
Moreover it holds for a general central force field $\vect{F}$, 
that is a force field of the form
\[ \vect{F}(\vect{r})=f(\vect{r})\vect{r}/r   \]
with $f$ a scalar function on $\mathbb{R}^3$ minus the origin.
Writing $\vect{p}=\mu\dot{\vect{r}}$ for the momentum vector 
it follows from the Leibniz product rule that the \emph{angular momentum} 
vector $\vect{L}=\vect{r}\times\vect{p}$ is conserved, which in case 
$\vect{L}\neq\vect{0}$ implies that the motion takes place in the plane 
perpendicular to $\vect{L}$.
Since the area of the surface traced out by the radius vector $\vect{r}$ in
a time interval $t_0<t<t_1$ is equal to
\[  \tfrac12\,\int_{t_0}^{t_1}|\vect{r}\times\dot{\vect{r}}|\,dt=
    L(t_1-t_0)/(2\mu) \]
we conclude that the radius vector $\vect{r}$ in a central force field 
sweeps out equal areas in equal times.

If the central force field $\vect{F}$ is in addition spherically symmetric,
that is 
\[  \vect{F}(\vect{r})=f(r)\vect{r}/r  \]
with $f$ a scalar function on $\mathbb{R}_+$, then the potential function
$V$ is defined by
\[ V(r)=-\int\,f(r)dr \]
and satisfies $d\{V(r)\}/dt=-f(r)(\vect{r}\cdot\dot{\vect{r}})/r$ by the chain rule.  
In turn this implies that the \emph{total energy} 
\[ H=p^2/(2\mu)+V(r)  \]
is conserved for solutions of Newton's equation of motion. For the
Newtonian force field $f(r)=-k/r^2$ the potential function
becomes $V(r)=-k/r$. 

\section{A geometric focus-focus proof}
\label{A geometric focus-focus proof}

In this section an ellipse will be the geometric locus of points
in a plane for which the sum of the distances to two given points
is constant. The two given points are called the foci, and the 
sum of the distances is denoted $2a$ and called the major axis. 
The distance between the given foci is denoted $2c$, and 
$2b>0$, defined by $a^2=b^2+c^2$, is called the minor axis.
The quotient $0\leq e=c/a\leq1$ is called the eccentricity of the 
ellipse.  If $e=0$ then the ellipse becomes a circle, while if 
$e=1$ then the ellipse degenerates to a line segment.

Let us continue the discussion at the end of the previous section,
and let us assume throughout this section that both 
$\vect{L}\neq\vect{0}$ (excluding collinear motion) and $H<0$ are fixed.
Consider the following figure of the plane perpendicular to $\vect{L}$.
The circle $\mathcal{C}$ with center $\vect{0}$ and radius $-k/H>0$ is 
the boundary of a disc where motion with fixed energy $H<0$ can take 
place. Indeed, we have
\[ H=p^2/(2\mu)-k/r\geq -k/r  \]
and so $r\leq -k/H$ with equality if and only if $p=0$. 
The solutions $t\mapsto\vect{r}$ of the Kepler problem starting from
rest at points of $\mathcal{C}$ fall straight onto the origin $\vect{0}$.
For this reason $\mathcal{C}$ is called the \emph{fall circle}
\cite{Thomas}.

Let $\vect{s}=-k \vect{r}/(rH)$ be the projection of $\vect{r}$ from 
the center $\vect{0}$ on this circle $\mathcal{C}$. 
The line $\mathcal{L}$ through $\vect{r}$ with direction vector
$\vect{p}$ is the tangent line to the orbit $\mathcal{E}$ at position 
$\vect{r}$ with momentum $\vect{p}$. Let $\vect{t}$ be the orthogonal 
reflection of the point $\vect{s}$ in the tangent line $\mathcal{L}$.
As time varies, the position vector $\vect{r}$ moves along the 
orbit $\mathcal{E}$ and also $\vect{p}=\mu\dot{\vect{r}}$ and 
$\mathcal{L}$ move along with it, and likewise the point $\vect{s}$ 
moves along the fall circle $\mathcal{C}$. It is a good question to 
investigate how the point $\vect{t}$ moves.

\begin{center}
\includegraphics{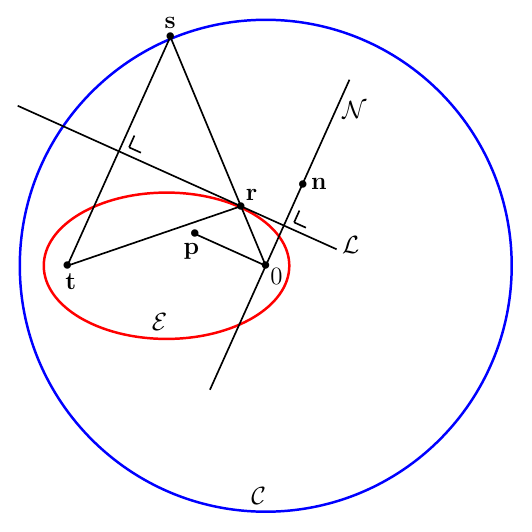}
\end{center}

\begin{theorem}\label{tconservedtheorem}
The point $\vect{t}$ is equal to $\vect{K}/(\mu H)$ with
\[ \vect{K}=\vect{p}\times\vect{L}-k\mu\vect{r}/r \]
the so called \emph{Lenz vector}. The Lenz vector $\vect{K}$ and 
therefore also the vector $\vect{t}$ are conserved quantities for the 
Kepler problem.
\end{theorem}

\begin{proof}
The line $\mathcal{N}$ spanned by $\vect{n}=\vect{p} \times \vect{L}$
is perpendicular to $\mathcal{L}$.
The point $\vect{t}$ is obtained from $\vect{s}=-k \vect{r}/(rH)$ by
subtracting twice the orthogonal projection of $\vect{s}-\vect{r}$ on 
the line $\mathcal{N}$, and therefore
\[ \vect{t}=\vect{s}-2((\vect{s}-\vect{r})\cdot\vect{n})\vect{n}/n^2. \]
Using
$ \vect{u}\cdot(\vect{v}\times\vect{w})=(\vect{u}\times\vect{v})\cdot\vect{w}$
for all vectors $\vect{u},\vect{v},\vect{w}$ in $\mathbb{R}^3$ we get
\[ 2(\vect{r}-\vect{s})\cdot\vect{n}=
2(H+k/r)\vect{r}\cdot(\vect{p}\times\vect{L})/H=p^2L^2/(\mu H) \]
and since $n^2=p^2L^2$ we conclude that
\[ \vect{t}=-k\vect{r}/(rH)+\vect{n}/(\mu H) = \vect{K}/(\mu H) \]
with $\vect{K} = \vect{p}\times\vect{L}-k\mu\vect{r}/r$ the Lenz vector.
The second claim that $\dot{\vect{K}}=0$ is derived by a straightforward 
computation using the Leibniz product rule for differentiation, and is 
left to the reader as an exercise.
\end{proof}

\begin{corollary}\label{ellipsetheorem} 
The orbit $\mathcal{E}$ is an ellipse with foci $\vect{0}$ and $\vect{t}$, 
and major axis equal to $2a=-k/H$.
\end{corollary}

\begin{proof}
Since orthogonal reflections preserve lengths we have
\[ \abs{\vect{t}-\vect{r}}+\abs{\vect{r}-\vect{0}} =
   \abs{\vect{s}-\vect{r}}+\abs{\vect{r}-\vect{0}} =
   \abs{\vect{s}-\vect{0}} = -k/H. \]
Hence $\mathcal{E}$ is an ellipse with foci $\vect{0}$ and $\vect{t}$, 
and with major axis $2a=-k/H$.
\end{proof}

This geometric proof of the law of ellipses is taken from
\cite{van Haandel--Heckman 2009}. The conserved vector 
$\vect{t} = \vect{K}/\mu H$ is a priori well motivated both in geometric 
and physical terms. In most text books on classical mechanics, 
like the one by Herbert Goldstein \cite{Goldstein 1980}, 
or in the original article by Wilhelm Lenz \cite{Lenz 1924} the vector 
$\vect{K}$ is just written down out of the blue and its motivation 
comes only a posteriori from the conservation law 
$\dot{\vect{K}}=\vect{0}$ and as a vector pointing in the direction 
opposite to the focus $\vect{t}$ of the elliptical orbit $\mathcal{E}$.

The vector $\vect{K}$ has been (re)discovered many times before, going 
back to Hermann and Laplace and others
\cite{van Haandel--Heckman 2009}. In the literature it is 
commonly called the Runge--Lenz vector, or also just the Lenz vector.
Pauli introduced a quantized version of the Lenz vector to give an  
elegant derivation of the Balmer formulae for the hydrogen spectrum 
\cite{Pauli 1926}, \cite{Van der Waerden 1968}. Pauli did this work 
in the fall of 1925 at Hamburg, where he was assistent with Lenz.

By definition we find $e=2c/(2a)=-K/(\mu H):-k/H=K/(k\mu)$ for the 
eccentricity of $\mathcal{E}$. The square length of the Lenz vector
 is equal to
\[ \vect{K}\cdot\vect{K}=
   (\vect{p}\times\vect{L})\cdot(\vect{p}\times\vect{L})
   -2(\vect{p}\times\vect{L})\cdot(k\mu\vect{r}/r)+k^2\mu^2=
   2\mu HL^2+k^2\mu^2 \]
by straightforward inspection. If $2c$ is the distance between the 
two foci of the elliptical orbit $\mathcal{E}$ then
\[ 4c^2=\vect{t}\cdot\vect{t}=(2\mu HL^2+k^2\mu^2)/(\mu^2H^2) \]
and together with $4a^2=4b^2+4c^2=k^2/H^2$ we arrive at 
$4b^2=-2L^2/(\mu H)$.

The area of the region bounded inside $\mathcal{E}$
is $\pi ab$, and therefore
\[ \pi ab=LT/2\mu \]
with $T$ the period of the orbit. Hence we obtain
\[ \frac{a^3}{T^2}=\frac{aL^2}{4\pi^2b^2\mu^2}=
   \frac{-2ab^2\mu H}{4\pi^2b^2\mu^2}=
   \frac{k}{4\pi^2\mu}=\frac{\mathrm{G}(m+M)}{4\pi^2} \]
using $k=\mathrm{G}mM$ and $\mu=mM/(m+M)$. Since the mass $m$ of
any planet is negligible compared to the mass $M$ of the sun we
conclude that the ratio $a^3/T^2$ is the same for all planets,
which is how Kepler formulated his harmonic law. 
This ends our discussion of the three Kepler laws: the ellipse law,
the area law and the harmonic law.

\section{All Kepler ellipses through a fixed point}
\label{All Kepler ellipses through a fixed point}

Let us continue with the notation of the previous section, that is
let us fix an energy $H=p^2/2\mu-k/q<0$ and let $\mathcal{C}$ be 
the falling circle with center at the origin $\vect{0}$ and radius 
$-k/H$. Otto Laporte asked himself the question what can be said
about the one parameter family $\mathcal{F}$ of all Kepler ellipses 
$\mathcal{E}$ having that same fixed energy $H<0$ and passing 
through a fixed point $\vect{r}$ in space \cite{Laporte 1970}. 
Our geometric approach for the Kepler problem answers these 
questions rather easily.

For example, what is the locus $\mathcal{T}$ of the foci 
$\vect{t}$ as these Kepler ellipses through the fixed point 
$\vect{r}$ vary? The geometry gives a quick answer, because
\[ |\vect{t}-\vect{r}|= |\vect{s}-\vect{r}|=s-r=-k/H-r=2a-r \]
and so $\vect{t}$ traverses a circle with center $\vect{r}$
and radius $2a-r$. The ellipse in this one parameter family
with smallest eccentricity $e=1+2Hr/k$ is the one with $\vect{r}$ 
at its perihelion and $\vect{r}-\vect{s}$ at its aphelion, 
while the one with largest eccentricity $e=1$ is the fall
from standstill at $\vect{s}$ reaching $\vect{0}$ in finite time $T/2$ 
with infinite velocity. Indeed, all Kepler ellipses with the same energy 
$H$ have the same major axes $2a=-k/H$, and hence also the same 
period $T$ by the harmonic law. In particular all motions starting  
at $\vect{r}$ at the same time return at $\vect{r}$ simultaneously.

\begin{center}
\includegraphics{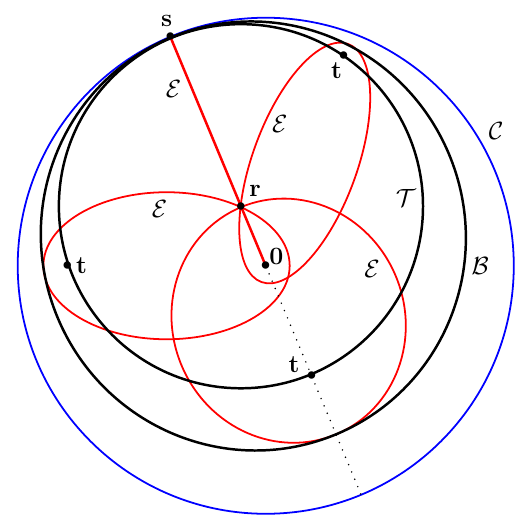}
\end{center}

Another question that Laporte posed is to describe the locus 
$\mathcal{B}$ of points that bounds the region swept out by all 
Kepler ellipses through the fixed point $\vect{r}$. If $\vect{q}$ 
is a point on such an ellipse $\mathcal{E}$ with focus $\vect{t}$ 
then
\[ q+|\vect{q}-\vect{r}|\leq q+|\vect{q}-\vect{t}|+
   |\vect{t}-\vect{r}|+r-r=-k/H-k/H-r=4a-r \]
by the triangle inequality, and equality holds if $\vect{t}$ lies
on the line segment from $\vect{q}$ to $\vect{r}$. Hence the region
swept out by these Kepler ellipses through $\vect{r}$ with energy 
$H<0$ is bounded by an ellipse with foci $\vect{0}$ and $\vect{r}$ 
and major axis $4a-r$.  

His last question deals with the directrices of $\mathcal{E}$
with respect to the origin, as $\mathcal{E}$ varies in the family
$\mathcal{F}$ of Kepler ellipses through the fixed point $\vect{r}$. 
The directrix $\mathcal{D}$ of such an $\mathcal{E}$ with respect 
to the origin is given by $\vect{d}+\vect{K}^{\perp}$ with
$\vect{d}=L^2\vect{K}/K^2$ and $\vect{K}^{\perp}$ the orthogonal
complement of $\vect{K}$. Indeed the distance from $\vect{r}$
to this directrix $\mathcal{D}$ is equal to
\[ (\vect{d}-\vect{r})\cdot\vect{K}/K=
   (L^2-\vect{r}\cdot\vect{K})/K=k\mu r/K=r/e \]
with $e=K/(k\mu)$ the eccentricity of $\mathcal{E}$, as should.   
The degenerate ellipse $\mathcal{E}$ through $\vect{r}$ with 
maximal eccentricity $e=1$ has directrix equal to $\vect{r}^{\perp}$ 
while the ellipse $\mathcal{E}$ through $\vect{r}$ with minimal 
eccentricity $e=1+2Hr/k=(a-r)/a$ has directrix equal to
\[ (1+a/(a-r))\vect{r}+\vect{r}^{\perp}  \]
at least if $r\neq a$. 

Let us assume for the rest of this section that $0<r<a$ , 
which in turn implies that the complement of the region swept out by
this family of directrices is bounded. Let $\mathbf{E}$ denote the
curve bounding that bounded complement. A natural Ansatz would 
be that $\mathbf{E}$ is an ellipse with foci $\vect{r}$ and 
$\vect{u}=a\vect{r}/(a-r)$ and with long axis equal to $(2a-r)r/(a-r)$.

\begin{center}
\includegraphics{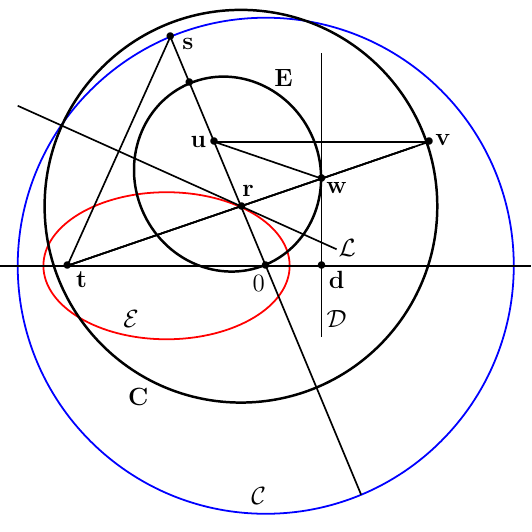}
\end{center}

\begin{theorem}
The orthogonal reflection of the vector $\vect{u}=a\vect{r}/(a-r)$
in the directrix $\mathcal{D}=\vect{d}+\vect{K}^{\perp}$ of the
Kepler ellipse $\mathcal{E}$ through $\vect{r}$ is equal to 
\[ \vect{v}=a\vect{r}/(a-r)-r\vect{t}/(a-r),  \]
which in turn implies that $\vect{v}-\vect{r}=r(\vect{r}-\vect{t})/(a-r)$.
In particular we get $|\vect{v}-\vect{r}|=(2a-r)r/(a-r)$ and so
$\vect{v}$ moves along a circle $\mathbf{C}$ with center $\vect{r}$ 
and radius $(2a-r)r/(a-r)$ as $\mathcal{E}$ moves in the family 
$\mathcal{F}$ of Kepler ellipses through $\vect{r}$. 
Hence this family of directrices of $\mathcal{E}$ is the family of tangents 
to an ellipse $\mathbf{E}$ with foci $\vect{r}$ and $\vect{u}=a\vect{r}/(a-r)$, with long axis equal to $(2a-r)r/(a-r)$ and with eccentricity 
$[r^2/(a-r)]:[(2a-r)r/(a-r)]=r/(2a-r)$.
\end{theorem}

\begin{proof}
The orthogonal reflection $\vect{v}$ of $\vect{u}=a\vect{r}/(a-r)$ 
with mirror the directrix $\mathcal{D}=\vect{d}+\vect{K}^{\perp}$ 
is given by the formula
\[ \vect{v}=\vect{u}-2((\vect{u}-\vect{d})\cdot\vect{K})\vect{K}/K^2,  \]
and the desired rewriting goes as follows. Since
\[ \vect{u}\cdot\vect{K}=a\vect{r}\cdot\vect{K}/(a-r)=a(L^2-k\mu r)/(a-r),\;
\vect{d}\cdot\vect{K}=L^2 \]
we get
\[ 2((\vect{u}-\vect{d})\cdot\vect{K})\,\vect{K}=
   2r(L^2-ak\mu)\,\vect{K}/(a-r)=rK^2\,\vect{t}/(a-r)\,. \]
Here we have used
\[ 2a=-k/H,\;\vect{K}=\mu H\vect{t},\;K^2=2\mu HL^2+k^2\mu^2\,.\]
This proves that $\vect{v}=a\vect{r}/(a-r)-r\vect{t}/(a-r)$ and hence
we conclude that $|\vect{v}-\vect{r}|=r(2a-r)/(a-r)$. The rest of the
theorem follows just like the argument of the previous section.
\end{proof}

\begin{remark}
The ellipse $\mathbf{E}$ has eccentricity $r/(2a-r)$ and
so its directrix $\mathbf{D}$ with respect to the focus
$\vect{r}$ is equal to $-(2a-r)\vect{r}/r+\vect{r}^{\perp}$.
This suggests that in case $r=a$ the dual curve $\mathbf{E}$ 
becomes a parabola, and in case $a<r<2a$ the dual curve 
$\mathbf{E}$ becomes a hyperbola. We leave it to the interested
reader to show that the above geometric argument can be adapted
to include these cases as well.
\end{remark}

\section{Final remarks}
\label{Final remarks} 

In the fall of 1995 I spent a month at the Mittag Leffler Institute in 
Stockholm. In the impressive library I was brousing through the 
correspondences of G\"{o}sta Mittag Leffler with Sophie Kowalevski 
about her discovery of the famous integrable top, and later went 
down to the basement of the Institute to get myself a reprint of 
her Acta paper from 1889 \cite{Kowalevski 1889}. 
Motivated by my previous work with Eric 
Opdam on hypergeometric functions associated with root systems 
(which was partly motivated by understanding how the integrals of 
motion for the classical Calogero--Moser system could be lifted
to its quantization) I checked by trial and error that her classical 
integral of motion could be lifted to a conserved quantity for the 
corresponding quantum top, and wrote a short paper with the
algebraic details of the proof \cite{Heckman 1998}.

In 2005 I got a friendly letter of the Russian physicist Igor Komarov,
explaining that both the quantum integrability of the Kowalevski top
had been done long before in 1933 by Otto Laporte \cite{Laporte 1933},
and also that my approach by doing the calculations in the universal 
enveloping algebra of the Euclidean motion group of $\mathbb{R}^3$ 
had been anticipated by him in 1981 \cite{Komarov 1981} with several 
related results in the following years \cite{Komarov 2001}.
I should have written back then a short addendum to my paper
explaining my ignorance of this earlier work by Laporte and 
Komarov, but postponed this idea with the plan of
getting back to the quantum Kowalevski top and see if some better 
understanding of the corresponding spectral problem could be 
obtained. 

It did not work out that way as I failed in this attempt, and later I 
forgot about it, until I read a few years ago the autobiography 
``Der Teil und das Ganze'' of Werner Heisenberg. 
In Chapter 3 Heisenberg tells about his contacts with Wolfgang 
Pauli and Otto Laporte, which revitalized my interest in the person 
of Laporte. All three were graduate students of Arnold Sommerfeld 
in M\"{u}nchen with graduation years 1921 (P), 1923 (H) and 1924 (L).
Subsequently Laporte went as a postdoc to the National Bureau of 
Standards in  Maryland. In 1926 he joined the physics faculty at Ann
Arbor in Michigan as colleague of Sam Goudsmit and George Uhlenbeck,
and stayed there for the rest of his life. The paper on the Kepler 
ellipses through a fixed point in space of 1970 was one of his last,
written after many years of teaching classical mechanics. 
By shining some extra light now on this Kepler paper of Laporte  
I hope to have made up for the omission in my old work of 1998.

\noindent
Gert Heckman, Radboud University Nijmegen: g.heckman@math.ru.nl


\begin{thebibliography}{25}

\bibitem{Givental 2016}
Alexander Givental, Kepler's Laws and Conic Sections,
Arnold Math. Journal {\bf 2} (2016), 139-148.

\bibitem{Goldstein 1980}
Herbert Goldstein, Classical Mechanics, Addison-Wesley, 
Second Edition, 1980.

\bibitem{Goodstein--Goodstein 1996}
David L. Goodstein and Judith R. Goodstein,
Feynman's Lost Lecture: The Motion of Planets Around
the Sun, Norton and Company, New York, 1996.

\bibitem{van Haandel--Heckman 2009}
Maris van Haandel and Gert Heckman, Teaching the Kepler Laws for
Freshmen, The Mathematical Intelligencer {\bf 31}:2 (2009), 40-44. 

\bibitem{Heckman 1998}
Gert Heckman, Quantum Integrability of the Kovalevsky Top,
Indagationes Mathematicae {\bf 9}:3 (1998), 359-365.

\bibitem{Heisenberg 1969}
Werner Heisenberg, Der Teil und das Ganze, Piper Verlag, 
Berlin, 1969. 

\bibitem{Komarov 1981}
Igor V. Komarov, Kovalevskaya basis for the hydrogen atom,
Theor. Math. Phys. {\bf 47} (1981), 76-72.

\bibitem{Komarov 2001}
Igor V. Komarov, Remarks on Kowalevski's top, J. Phys. A {\bf 34}:11
(2001), 2111-2120.

\bibitem{Kowalevski 1889}
Sophie Kowalevski, Sur le Probl\`{e}me de la Rotation d' un Corps
Solide autour d'un Point Fixe, Acta Math. {\bf 12}:1 (1889), 177–232, 

\bibitem{Laporte 1933}
Otto Laporte, Note on Kowalewski's Top in Quantum Mechanics,
Physical Review {\bf 43} (1933), 548-551.

\bibitem{Laporte 1970}
Otto Laporte, On Kepler Ellipses Starting from a Point in Space, 
American Journal of Physics {\bf 38}:7 (1970), 837-840.

\bibitem{Lenz 1924}
Wilhelm Lenz, \"{U}ber den Bewegungsverlauf und Quantenzust\"{a}nde der 
gest\"{o}rten Keplerbewegung, Zeitschrift f\"{u}r Physik {\bf 24} 
(1924), 197-207.

\bibitem{Newton 1687}
Isaac Newton, The Principia, Mathematical Principles of Natural
Philosophy, New Translation by I. Bernard Cohen and Anne Whitman,
University of California Press, Berkeley, 1999. 

\bibitem{Pauli 1926}
W. Pauli, \"{U}ber das Wasserstoffspektrum vom Standpunkt der neuen
Quantenmechanik, Zeitschrift f\"{u}r Physik {\bf 36} (1926), 336-363.

\bibitem{Van der Waerden 1968}
B.L. van der Waerden, Sources of Quantum Mechanics,
Dover Publications, New York, 1968.

\bibitem{Thomas}
Oswald Thomas, Astronomie – Tatsachen und Probleme, 
Das Bergland-Buch, Salzburg, 1949.

\end{thebibliography}
\end{document}